\documentclass[reqno]{amsart}

\usepackage{amsfonts}
\usepackage{amssymb}
\usepackage{amsmath}
\usepackage{mathrsfs}
\usepackage{amsthm}
\usepackage[colorlinks,linkcolor=blue,citecolor=blue]{hyperref}

\newtheorem{theorem}{Theorem}[section]

\newtheorem{lemma}[theorem]{Lemma}
\newtheorem{proposition}[theorem]{Proposition}

\theoremstyle{remark}
\newtheorem{remark}[theorem]{Remark}

\theoremstyle{definition}

\numberwithin{equation}{section}

\begin{document}
	
	\title[A quantitative symmetry result for $p$-Laplace equations]{A quantitative symmetry result for $p$-Laplace equations with discontinuous nonlinearities}

	\author[G. Ciraolo]{Giulio Ciraolo}
	\address[G. Ciraolo]{Dipartimento di Matematica ``Federigo Enriques"\\
		Universit\`a degli Studi di Milano\\ Via Cesare Saldini 50, 20133 Milan\\
		Italy}
	\email{giulio.ciraolo@unimi.it}
	
	\author[X. Li]{Xiaoliang Li}
	\address[X. Li]{Dipartimento di Matematica ``Federigo Enriques"\\
		Universit\`a degli Studi di Milano\\ Via Cesare Saldini 50, 20133 Milan\\
		Italy}
	\email{xiaoliang.li@unimi.it}
	
	\thanks{}
	
	
	\begin{abstract}
	In this paper, we study positive solutions $u$ of the homogeneous Dirichlet problem for the $p$-Laplace equation $-\Delta_p \,u=f(u)$ in a bounded domain $\Omega\subset\mathbb{R}^N$, where $N\ge 2$, $1<p<+\infty$ and $f$ is a discontinuous function. We address the quantitative stability of a Gidas--Ni--Nirenberg type symmetry result for $u$, which was established by Lions \cite{Lions81} and Serra \cite{Ser13} when $\Omega$ is a ball. By exploiting a quantitative version of the P\'olya--Szeg\"o principle, we prove that the deviation of $u$ from its Schwarz symmetrization can be estimated in terms of the isoperimetric deficit of $\Omega$. 
	\end{abstract}

	\maketitle
	
	\section{Introduction}
	Given a bounded domain $\Omega$ in $\mathbb{R}^N$ with $N\ge 2$, we consider the problem
	\begin{equation}\label{in-eq:pro}
		\begin{cases}
			-\Delta_p \,u=f(u) & \text{in }\Omega\\
			u>0 & \text{in }\Omega\\
			u=0 & \text{on }\partial\Omega,
		\end{cases}
	\end{equation}
	where $\Delta_p \, u=\mathrm{div}\,(|\nabla u|^{p-2}\nabla u)$ denotes the $p$-Laplacian operator, $p>1$, and $f$ is a given function on $[0,+\infty)$. 

	It is well known that solutions to \eqref{in-eq:pro} could be expected to be of spherical symmetry when $\Omega$ is a ball (say $\Omega=B$), under suitable assumptions on the nonlinearity $f$.

	This study was pioneered by Gidas, Ni and Nirenberg in their seminal paper \cite{GNN79}. They proved that, for the Laplacian case $p=2$, every classical solution to \eqref{in-eq:pro} is radially symmetric and decreasing, provided that $\Omega=B$ and $f=f_1+f_2$ with $f_1$ Lipschitz and $f_2$ non-decreasing.

	Subsequently, such a symmetry result has been generalized to the quasi-linear case $p\neq 2$ in various works, such as in \cite{DP98,DP00,DS04} for a Lipschitz function $f$, and in \cite{KP94,Bro97,DFM05} for a  continuous $f$ not necessarily Lipschitz.

Symmetry results have also been developed when $f$ may be discontinuous.
The contribution in this direction was first due to Lions \cite{Lions81}, who showed the radial symmetry and monotonicity of solutions to \eqref{in-eq:pro} under the conditions that $\Omega=B$, $p=N=2$ and $f$ is positive and locally bounded. Serra \cite{Ser13} later extended this result to hold for \eqref{in-eq:pro} with any $p>1$ and $N\ge 2$, though there additional growth assumptions on $f$ are required when $1<p<N$.

	It should be pointed out that the techniques used in the above mentioned papers are different. The proofs in Gidas--Ni--Nirenberg \cite{GNN79} as well as in \cite{DP98,DP00,DS04,DFM05} are based on the method of moving planes (in Dolbeault--Felmer--Monneau \cite{DFM05}, a local version of this method was exploited), while those in Brock \cite{Bro97} carry out the so-called ``continuous Steiner Symmetrization" procedure (see also \cite{Bro00}). Lions' method in \cite{Lions81} combines the isoperimetric inequality and the Pohozaev identity, and this argument was generalized in Kesavan--Pacella \cite{KP94} to high dimensional cases $p=N$ and further in Serra's paper \cite{Ser13} to the case of any $p\neq N$.

    These qualitative studies naturally prompt the stability question of whether a solution to problem \eqref{in-eq:pro} ``almost" remains radially symmetric under small ``perturbations" of certain conditions leading to symmetry, such as the sphericity of the domain $\Omega$. However, few results about this issue are available in literature. To the best of our knowledge, the first effort to address this problem was made by Rosset \cite{Ros94}, who obtained an approximate Gidas--Ni--Nirenberg result illustrating in a quantitative way that a classical solution to \eqref{in-eq:pro} (with $p=2$ and $f$ being locally Lipschitz) is close to a radial function if $\Omega$ is a $C^1$ perturbation of a ball.

	Recently, another quantitative version of the Gidas--Ni--Nirenberg result was established by the first author, Cozzi, Perugini and Pollastro \cite{CCPP24}. For problem \eqref{in-eq:pro} where $\Omega=B$, $p=2$ and $f$ is locally Lipschitz and nonnegative, they derived approximate symmetry and monotonicity of its solutions by considering perturbations of the nonlinearity $f$, as well as smooth perturbations of the domain.

	We mention that the proofs in \cite{Ros94} and \cite{CCPP24} are both based on a quantitative refinement of the method of moving planes for the Laplacian.

	The aim of this paper is to address the stability of symmetry properties for solutions to problem \eqref{in-eq:pro}, also in the quasi-linear case $p\neq 2$, under low regularity assumptions on $f$ and $\Omega$. We investigate a quantitative counterpart of Serra's result in \cite{Ser13} in the spirit of characterizing how closely solutions are from being symmetric when the domain $\Omega$ is nearly a ball in some sense (symboled below by $\Omega\approx B$).

More precisely, as considered in \cite{Ser13}, we assume $p\in(1,+\infty)$, $\Omega$ Lipschitz and $f$ possibly discontinuous, and we establish a quantitative result for problem \eqref{in-eq:pro} by estimating the $L^1$-distance between a solution to \eqref{in-eq:pro} and its \emph{Schwarz symmetrization} in terms of \emph{the isoperimetric deficit} of $\Omega$. 
	
	
	
	While preparing this manuscript, we learned that a similar work was recently done by Dipierro, Silva, Poggesi and Valdinoci \cite{DSPV24}. They also obtained an approximate symmetry result of $L^1$ type for problem \eqref{in-eq:pro} when $p\in(1,+\infty)$ and $\Omega\approx B$. We emphasize that there are relevant differences between \cite{DSPV24} and the present paper, in particular regarding the assumptions on $\Omega$ and $f$, the deficit to measure the proximity $\Omega\approx B$, and the approach of the proofs.

	Indeed, in \cite{DSPV24} problem \eqref{in-eq:pro} was studied by requiring that $\Omega$ is of class $C^2$ and $f$ is a locally Lipschitz function. The deficit in \cite{DSPV24} measuring $\Omega\approx B$ was introduced as the sum of the isoperimetric deficit of $\Omega$ and a quantity concerning integrals over $\partial\Omega$. Regarding the proof in \cite{DSPV24}, it exploits the arguments developed by Cianchi--Esposito--Fusco--Trombetti \cite{CEFT08} in the study of a quantitative P\'olya--Szeg\"o principle, and it also relies on a refinement of the regularity estimates derived in Damascelli--Sciunzi \cite{DS04}.

	Instead, as we will describe later in Subsection \ref{sub-sec:strategy}, our approach takes inspiration from the papers by Figalli--Maggi--Pratelli \cite{FMP2010} and Amato--Barbato--Masiello--Paoli \cite{ABMP23} in the study of sharp quantitative isoperimetric inequalities and of a quantitative Talenti comparison result, respectively. It enables us to deal with the symmetry issue for problem \eqref{in-eq:pro} under quite weak regularity assumptions on the domain and on the nonlinearity.

	To explicitly state our result, we introduce some notations. Given a bounded Lipschitz domain $\Omega\subset\mathbb{R}^N$, let $\delta_\Omega$ be the isoperimetric deficit of $\Omega$ defined in \eqref{pre-eq:iso-def}, and let $C_\Omega$ be the constant appearing in a Poincar\'e trace inequality (see \eqref{pre-eq:Poincare-trace}). Let $u^*$ be the Schwarz symmetrization of a measurable function $u$, as recalled in \eqref{pre-eq:ustar}. Moreover, if $u\in C(\overline{\Omega})$ is a weak solution to \eqref{in-eq:pro} with $f$ locally bounded, we let
$$
M=\max_{\overline{\Omega}} u \,, \quad m_f=\inf_{t\in(0,M]} f(t)
$$ 
and 
$$
M_f \textmd{ be such that } f(t)\le M_f \textmd{ for any } t\in[0,M] \,; 
$$
in this case, we agree to regard $u$ as a function in $\mathbb{R}^N$ by extending it to be $0$ outside $\Omega$.

   Our main result is the following. 
     \begin{theorem}\label{thm:main}
   	Let $\Omega$ be a bounded Lipschitz domain in $\mathbb{R}^N$ with $N\ge 2$, and let $f$ be a nonnegative Borel function on $[0,+\infty)$ that is locally bounded. Let $u\in C^1(\overline{\Omega})$ be a weak solution to \eqref{in-eq:pro}. 
   	
   Assume $m_f>0$ and that one of the following two conditions is fulfilled: 

\begin{equation} \label{pgeN}
p\ge N
\end{equation}
or 
\begin{equation}\label{in-eq:conditions-f}
1<p<N \,\text{  and  }\, \phi\le f\le \frac{Np-s}{N-p}\phi 
\end{equation}
for some nonincreasing function $\phi\ge 0$ and number $s>0$. 
   	 
   	Then there exist a constant $\rho$ depending only on $N$ and $p$, and a constant $C$, depending only on $N,p,|\Omega|,C_\Omega,M,M_f,m_f$ and $\|\nabla u\|_{L^\infty(\partial\Omega)}$ (and also on $s$ if $1<p<N$), such that
   		\begin{equation}\label{in-eq:u-ustar}
   			\inf_{x_0\in\mathbb{R}^N}\int_{\mathbb{R}^N}|u(x)- u^*(x+x_0)|\,dx\le C (\delta_\Omega)^{\rho}.
   		\end{equation}
   	\end{theorem}

   \begin{remark}
   	When $\Omega$ is a ball in $\mathbb{R}^N$, we have $\delta_\Omega=0$ and it follows from \eqref{in-eq:u-ustar} that the solution $u$ coincides with $u^*$ up to a translation, thus implying $u$ is radially symmetric and decreasing. This recovers Serra's result obtained in \cite[Theorem 1]{Ser13}. Note that in \cite{Ser13}, the nonlinearity $f$ was considered in a more general form where it assumes \eqref{in-eq:conditions-f} to hold with $s=0$ and does not require $m_f>0$. In the quantitative case we are considering, the assumptions on $f$ are strengthened to overcome technical difficulties. To see this, we refer to Remark \ref{com-rk:C-s-mf} to gain insight into how the constant $C$ in \eqref{in-eq:u-ustar} depends on $m_f$ and $s$ in our arguments.
   \end{remark}
   
   \begin{remark}
   	Estimate \eqref{in-eq:u-ustar} can be boosted in the way of replacing the $L^1$ norm by $L^q$ norm with $q>1$ on the left-hand side. Actually, by exploiting the Gagliardo-Nirenberg inequalities, it is easy to infer from \eqref{in-eq:u-ustar} and \eqref{pf-eq:est-Dustar} that 
   	\begin{equation*}
   		\inf_{x_0\in\mathbb{R}^N}\int_{\mathbb{R}^N}|u(x)- u^*(x+x_0)|^q\,dx\le \tilde{C} (\delta_\Omega)^{\tilde{\rho}},
   	\end{equation*}
  for any $q\in [1, \frac{Np}{N-p})$ if $1<p<N$, any $q\ge 1$ if $p=N$, and  $q=\infty$ if $p>N$. Here, $\tilde{\rho}$ and $\tilde{C}$ are constants analogous to those in \eqref{in-eq:u-ustar}, but in this case, they also depend on $q$. 
   \end{remark}

   


   \subsection{Strategy of the proof and organization of the paper} \label{sub-sec:strategy}
   
   The starting point of the proof of Theorem \ref{thm:main} lies in retracing Serra's arguments in \cite{Ser13} to extract quantitative information on the symmetry of the solution. As already mentioned, the method used in \cite{Ser13} is a substantial extension of that developed in \cite{Lions81,KP94}, which is based on integral identities. To be precise, in the sequel we let $u\in C^1(\overline{\Omega})$ be a weak solution to \eqref{in-eq:pro} and introduce $\mathfrak{D}$ as the quantity
    \begin{equation}\label{in-eq:def-D}
   	\mathfrak{D}=\frac{N(p-1)}{ \kappa_N^{\frac{p}{p-1}}p}|\Omega|^{\frac{p-N}{N(p-1)}}\left[\int_\Omega f(u)\,dx\right]^{\frac{p}{p-1}}-N\int_\Omega F(u)\,dx-\frac{p-N}{p}\int_\Omega uf(u)\,dx,
   \end{equation}
   where $F(t)=\int _0^t f(\lambda)\,d\lambda$ and $\kappa_N$ is as in \eqref{pre-eq:iso}. The proof in \cite{Ser13} consists in two ingredients. One is to show $\mathfrak{D}\ge 0$ via a procedure applying the isoperimetric inequality and H\"older inequality on level sets of $u$, and the other is to exploit a Pohozaev-type identity to get instead $\mathfrak{D}\le 0$ when $\Omega$ is a ball. Thus, it yields $\mathfrak{D}=0$ in the case of balls, which in turn implies that those inequalities used to derive $\mathfrak{D}\ge 0$ are actually equalities. This forces $u$ towards being radially symmetric and decreasing.

   
    Refining the above demonstration, in the present paper we prove Theorem \ref{thm:main} with two steps. First, we shall show that the quantity $\mathfrak{D}$ can be bounded by the isoperimetric deficit of $\Omega$ (i.e. $\delta_\Omega$ above). Then we further demonstrate that, up to translations, the distance between $u$ and $u^*$ in $L^1(\mathbb{R}^N)$ can be estimated in terms of $\mathfrak{D}$. By combining these results we conclude the proof of Theorem \ref{thm:main}.

More precisely, the first step is treated in Section \ref{sec:est} where, in Proposition \ref{est-pro:D-iso},  we derive an explicit bound on $\mathfrak{D}$ in terms of the size of $\sqrt{\delta_\Omega}$. This is achieved by exploiting ideas from \cite{FMP2010} that involve \emph{the Brenier map} from mass transportation theory and the application of Poincar\'e-type trace inequalities.

To accomplish the second step, we develop a comparison result of Talenti type related to problem \eqref{in-eq:pro} and exploit a suitable version of the quantitative P\'olya--Szeg\"o principle proved in \cite{CEFT08} (see Theorem \ref{pre-thm:PZ}). This scheme is motivated by noting that the derivation of inequality $\mathfrak{D}\ge 0$ in \cite{Ser13} actually parallels that of the classical P\'olya--Szeg\"o principle recalled in \cite{CEFT08}, as well as that of Talenti's comparison theorems in \cite{Talenti76,Talenti79}. Indeed, these results are all derived from an application of the isoperimetric inequality and H\"older inequality on the level sets of the considered function and/or its Schwarz symmetrization and, essentially, their extremal cases yield that the applied isoperimetric inequality and H\"older inequality are both equalities (thus suggesting the symmetry of the extremals, as argued in \cite{BZ,Lions81,KP94,Ser13,ALT86}, for instance).

Accordingly, when investigating the asymmetry of $u$ under perturbations of the condition $\mathfrak{D}=0$, we are naturally connected to the quantitative counterpart of P\'olya--Szeg\"o principle studied in \cite{CEFT08} that quantifies the deviation of $u$ from $u^*$ under perturbations of the equality $\int_{\mathbb{R}^N} |\nabla u|^p\,dx=\int_{\mathbb{R}^N} |\nabla u^*|^p\,dx$ instead. Related to this, we mention that the result in \cite{CEFT08} has been recently used in \cite{ABMP23} to establish a quantitative version of the classical comparison result of Talenti \cite{Talenti76} for the Poisson equation.

Inspired by the discussions in \cite{ABMP23}, we derive in Section \ref{sec:com} a comparison result for $u$ in terms of solution to
\begin{equation}\label{in-eq:pro-v}
	\begin{cases}
		-\Delta_p \,v=f(u^*) &\text{in }\Omega^*\\
		v=0&\text{on }\partial\Omega^* ,
	\end{cases}
\end{equation}
where $\Omega^*$ is the open ball centered at the origin such that $|\Omega^*|=|\Omega|$. Indeed, it is proved in Theorem \ref{com-thm:Talenti} that $u^*\le v$ in $\Omega^*$ and that $\|v-u^*\|_{L^\infty(\Omega^*)}$ is bounded by the size of $\mathfrak{D}$. On the other hand, we are able to apply Theorem \ref{pre-thm:PZ} to estimate the $L^1$-distance between $u$ and $u^*$ via $\|v-u^*\|_{L^\infty(\Omega^*)}$, see Proposition \ref{pf-pro:u-ustar}. From these, the second step of our proof is completed.

The rest of the paper is organized as follows. In Section \ref{sec:pre} we provide several auxiliary results that will be used in our arguments. Section \ref{sec:est} is devoted to deriving an estimate on the quantity $\mathfrak{D}$ in terms of the isoperimetric deficit. Then, in Section \ref{sec:com}, we establish a comparison theorem for the symmetrized problem \eqref{in-eq:pro-v}. Finally, we prove Theorem \ref{thm:main} in Section \ref{sec:pf}.

   \section{Preliminaries} \label{sec:pre}
 
   Throughout the paper, we fix $D^c:=\mathbb{R}^N\setminus D$ for any measurable set $D\subset\mathbb{R}^N$, and denote by $\partial D$, $\chi_D$ and $|D|$ its topological boundary, characteristic function and $N$-dimensional Lebesgue measure, respectively. Let $B_r$ be the open ball in $\mathbb{R}^N$ centered at the origin with radius $r$, and let the symbol $\mathcal{H}^{N-1}$ stand for the $(N-1)$-dimensional Hausdorff measure.

   \subsection{The isoperimetric deficit and a Poincar\'e trace inequality}
   
   Hereafter, let $\Omega\subset\mathbb{R}^N$ be a bounded Lipschitz domain. Notice that the perimeter of $\Omega$ is given by $\mathcal{H}^{N-1}(\partial\Omega)$. Then the classical isoperimetric inequality states that 
   \begin{equation}\label{pre-eq:iso}
   	\mathcal{H}^{N-1}(\partial\Omega)\ge \kappa_N|\Omega|^{\frac{N-1}{N}},
   \end{equation}
   where $\kappa_N=N\omega_N^{\frac{1}{N}}$ with $\omega_N=|B_1|$. As introduced in \cite{FMP08,FMP2010} for the study of sharp quantitative isoperimetric inequalities, the isoperimetric deficit of $\Omega$, denoted by $\delta_\Omega$, is defined as
   \begin{equation}\label{pre-eq:iso-def}
   	\delta_\Omega=\frac{\mathcal{H}^{N-1}(\partial\Omega)}{\kappa_N|\Omega|^{\frac{N-1}{N}}}-1.
   \end{equation}
    It is well known that the equality case of \eqref{pre-eq:iso} (i.e. $\delta_\Omega=0$) holds if and only if $\Omega$ is a ball.

   In addition, for such $\Omega$, it is known that the following Poincar\'e trace inequality holds (see, for instance, \cite[Theorem 9.6.4]{Mayza11}):
   \begin{equation}\label{pre-eq:Poincare-trace}
   	\inf_{c\in\mathbb{R}}\|\mathrm{tr} (g)-c\|_{L^1(\partial\Omega)}\le C_\Omega \|Dg\|(\Omega)
   \end{equation}
   for some constant $C_\Omega$ and every $g\in BV(\Omega)$ (the space of functions of bounded variation in $\Omega$); here, $\mathrm{tr}(g)$ denotes the trace of $g$ on $\partial\Omega$ and $\|Dg\|(\Omega)$ stands for the total variation over $\Omega$ of the distributional gradient $Dg$ of $g$. Moreover, the infimum in \eqref{pre-eq:Poincare-trace} can be attained for some $c$, referring to \cite[Formula (1.2)]{CFNT17} and the references therein.

    \subsection{Rearrangements } \label{sub-sec:rearrange}
   
   Let $u$ be a measurable function in $\mathbb{R}^N$. Set 
   \begin{equation} \label{LambdaGamma_def}
   	\Lambda_u^t:=\{x\in\mathbb{R}^N: |u(x)|>t\},\quad \Gamma_u^{\tau}:=\{x\in\mathbb{R}^N: u(x)=\tau\},
   	\end{equation}
   for $t\ge 0$ and $\tau\in\mathbb{R}$; let $J_u(t)$ denote the distribution function of $u$,  i.e. 
   \begin{equation} \label{Ju_def}
  J_u(t)=|\Lambda_u^t| \,.
  \end{equation} 
  Assume that $u$ vanishes at infinity, that is, $J_u(t)$ is finite for all $t>0$, and recall that the \emph{decreasing rearrangement} $u^\#$ of $u$ and the \emph{Schwarz symmetrization} $u^*$ (or the \emph{symmetric decreasing rearrangement}) of $u$ are defined by
   \begin{equation} \label{usharp_def}
   	u^\#(s)=\sup \, \{t\ge0: J_u(t)>s\}\quad\text{for }s\ge 0,\\
   \end{equation}
   and 
   \begin{equation}\label{pre-eq:ustar}
   	u^*(x)=u^\#(\omega_N|x|^N)\quad\text{for }x\in \mathbb{R}^N,
   \end{equation}
   respectively. When $u\in L_{\mathrm{loc}}^1(\mathbb{R}^N)$, we will study the pointwise behavior of $u$ by assuming it to agree with its \emph{precise representative} $\tilde{u}$, which is given by  
   \begin{equation} \label{utilde_def}
   	\tilde{u}(x)=
   	\begin{cases}
   		\lim\limits_{r\to0}\frac{1}{|B_r(x)|}\int_{B_r(x)} u(y) \,dy & \text{if this limit exists}\\
   		0 & \text{otherwise}
   		\end{cases}
   \end{equation}  
   for $x\in\mathbb{R}^N$. In addition, denote by $\mathrm{ess\,sup}\, u$ the essential supremum of $u$, and let 
   \begin{equation} \label{Zu_def}
   Z_u:=\{x\in\mathbb{R}^N:\nabla u=0\},
   \end{equation}
   if $u$ is weakly differentiable.
   
   \begin{lemma} \label{pre-lem:prop-rea}
   	Assume $u\in W^{1,p}(\mathbb{R}^N)$, $1\le p<+\infty$. Then $u^*\in W^{1,p}(\mathbb{R}^N)$. Moreover, 
   	\begin{itemize}
   		\item[(i)] For any $t\in [0,\mathrm{ess\,sup}\, u]$, $u^{\#}(J_u(t))=t$ and it holds 
   		\begin{equation}\label{pre-eq:formula-Ju}
   			J_u(t)=J_{u^*}(t)=|Z_u\cap\Lambda_u^t|+\int_t^{\mathrm{ess\,sup}\, u}\int_{\Gamma_u^\tau}\frac{\chi_{Z_u^c}}{|\nabla u|} \,d\mathcal{H}^{N-1}\, d\tau.
   		\end{equation}
 \item[(ii)]  For almost all $t\in (0,\mathrm{ess\,sup}\, u)$, 
   		\begin{equation}\label{pre-eq:dJu}
   			J'_u(t):=\frac{d}{dt}J_u(t)=-\int_{\Gamma_{u^*}^t}\frac{\chi_{Z_{u^*}^c}}{|\nabla u^*|}\, d\mathcal{H}^{N-1} \le -\int_{\Gamma_u^t}\frac{\chi_{Z_u^c}}{|\nabla u|} \,d\mathcal{H}^{N-1}.
   		\end{equation}
   	
   	\item[(iii)] $J_u(t)$ is absolutely continuous on $(0, \mathrm{ess\,sup}\, u)$ if and only if $$|Z_{u^*}\cap \{x\in\mathbb{R}^N: 0<u^*<\mathrm{ess\,sup}\, u\}|=0.$$
    	\end{itemize}
   	\end{lemma}
   
   \begin{proof}
   	 These results are well-known in the literature, referring to Lemmas 2.3 and 2.4 in \cite{BZ} and \cite[Formula (2.4)]{CEFT08}, for instance.
   	\end{proof}
   
   
   \subsection{Quantitative P\'olya--Szeg\"o principle}
   As is well known, a classical principle going back to P\'olya--Szeg\"o \cite{PS51} states that 
   $$
   \int_{\mathbb{R}^N} |\nabla u|^p\,dx\ge \int_{\mathbb{R}^N} |\nabla u^*|^p\,dx
   $$
   for any $u\in W^{1,p}(\mathbb{R}^N)$, $1\le p<+\infty$. Concerning its quantitative counterpart, we have the following result obtained in \cite{CEFT08}. 
   
   \begin{theorem}\label{pre-thm:PZ}
   	Let $N\ge 2$ and $p>1$. There exist positive constants $r_1$, $r_2$, $r_3$ and $C$, depending only on $p$ and $N$, such that, for any $\delta>0$, $0<r\le\min\{r_1,r_2\}$ and $u\in W^{1,p}(\mathbb{R}^N)$ satisfying $J_u(0)<+\infty$, it holds
   	\begin{align}
   		&\min_{\pm}\inf_{x_0\in\mathbb{R}^N}\int_{\mathbb{R}^N}|u(x)\pm u^*(x+x_0)|\,dx \notag\\
   		\le &\, C\|\nabla u^*\|_{L^p(\mathbb{R}^N)}(J_u(0))^{1+\frac1N-\frac1p}\left[\mathcal{M}_{u^*}(\delta)+(\mathcal{E}_u)^r\left(1+\frac{\|\nabla u^*\|_{L^p(\mathbb{R}^N)}}{ (J_u(0))^{\frac1p}\delta}\right)\right]^{r_3}, \label{pre-eq:quant-PS}
   	\end{align}
   	with 
   	\begin{gather}
   		\mathcal{E}(u)=\frac{\int_{\mathbb{R}^N} |\nabla u|^p\,dx}{\int_{\mathbb{R}^N} |\nabla u^*|^p\,dx}-1, \label{pre-eq:def-E}\\
   	\mathcal{M}_{u^*}(\delta)=\frac{\left|\{x\in \mathbb{R}^N:|\nabla u^*(x)|\le\delta,\ 0<u^*(x)<\mathrm{ess\, sup}\, u\}\right|}{J_u(0)}\,, \label{pre-eq:def-M}
   		\end{gather}
and where $J_u$ and $u^*$ are given by \eqref{Ju_def} and \eqref{pre-eq:ustar}, respectively.
   \end{theorem}
   
   \begin{proof}
   	This is actually a modified version of \cite[Theorem 4.1]{CEFT08} which states that
   	\begin{align*}
   		&\min_{\pm}\inf_{x_0\in\mathbb{R}^N}\int_{\mathbb{R}^N}|u(x)\pm u^*(x+x_0)|\,dx\\
   		\le &\, C\|\nabla u^*\|_{L^p(\mathbb{R}^N)}(J_u(0))^{1+\frac1N-\frac1p}\left[\mathcal{M}_{u^*}(\delta)+(\mathcal{E}_u)^{r_1} +\frac{\|\nabla u^*\|_{L^p(\mathbb{R}^N)}}{ (J_u(0))^{\frac1p}\delta}(\mathcal{E}_u)^{r_2} \right]^{r_3}
   	\end{align*}
   	for some positive constants $r_1$, $r_2$, $r_3$ and $C$, depending only on $p$ and $N$. In view of this, it is clear that \eqref{pre-eq:quant-PS} is valid when $\mathcal{E}_u\le 1$, since $$(\mathcal{E}_u)^{r}\ge \max\{(\mathcal{E}_u)^{r_1} ,(\mathcal{E}_u)^{r_2} \} \quad\text{for any }0<r\le\min\{r_1,r_2\}.$$ 
   	
   	On the other hand, it follows from the Poincar\'e inequality and H\"older inequality that 
   	\begin{align*}
   		\min_{\pm}\inf_{x_0\in\mathbb{R}^N}\int_{\mathbb{R}^N}|u(x)\pm u^*(x+x_0)|\,dx \le &\, 2\|u^*\|_{L^1(\mathbb{R}^N)}\\
   		\le &\, 2\|u^*\|_{L^p(\mathbb{R}^N)}(J_u(0))^{1-\frac1p}\\
   		\le &\, 4\omega_N^{-\frac1N}\|\nabla u^*\|_{L^p(\mathbb{R}^N)}(J_u(0))^{1+\frac1N-\frac1p}.
   	\end{align*}
   	Hence, we see that \eqref{pre-eq:quant-PS} still holds true if $\mathcal{E}_u>1$, since
   	$$\mathcal{M}_{u^*}(\delta)+(\mathcal{E}_u)^r\left(1+\frac{\|\nabla u^*\|_{L^p(\mathbb{R}^N)}}{ (J_u(0))^{\frac1p}\delta}\right)>1$$
   	for any $r>0$. 
   	
   	The proof is complete.
   \end{proof}

   \section{Estimate on the quantity $\mathfrak{D}$ via the isoperimetric deficit} \label{sec:est}
   
   In this section, we prove the following proposition. Here and in the sequel, we adopt the notations introduced in the Introduction and Section \ref{sec:pre}. 
   
	\begin{proposition}\label{est-pro:D-iso}
			Let $p>1$ and let $\Omega$, $f$ and $u$ be as in Theorem \ref{thm:main}. Let $\mathfrak{D}$ and $\delta_\Omega$ be given by \eqref{in-eq:def-D} and \eqref{pre-eq:iso-def}, respectively.
			
			 If $\delta_\Omega\le 1$, then
			\begin{equation}\label{est-eq:D-iso}
				\mathfrak{D} \le \|\nabla u\|_{L^\infty(\partial\Omega)}^p\left(2^{\frac{1}{p-1}}N+9N^3C_\Omega\right)|\Omega|\sqrt{\delta_\Omega}\,,
			\end{equation}
			where $C_\Omega$ is as in \eqref{pre-eq:Poincare-trace}.
			\end{proposition}

	\begin{proof}
By the divergence theorem (see for instance \cite[Lemma 4.3]{CL22}) and H\"older inequality, we have
\begin{align*}
	\int_\Omega f(u)\,dx&=\int_{\partial\Omega}|\nabla u|^{p-1}\,d\mathcal{H}^{N-1}\\
	&\le[\mathcal{H}^{N-1}(\partial\Omega)]^{\frac{1}{p}}\left(\int_{\partial\Omega}|\nabla u|^p\,d\mathcal{H}^{N-1}\right)^{\frac{p-1}{p}}\\
	&=\left[N(\omega_N)^\frac{1}{N}|\Omega|^{\frac{N-1}{N}}(1+\delta_\Omega)\right]^{\frac{1}{p}}\left(\int_{\partial\Omega}|\nabla u|^p\,d\mathcal{H}^{N-1}\right)^{\frac{p-1}{p}}.
\end{align*}
Then
\begin{align}
	& \frac{N(p-1)}{p}(\kappa_N)^{-\frac{p}{p-1}} |\Omega|^{\frac{p-N}{N(p-1)}}\left[\int_\Omega f(u)\,dx\right]^{\frac{p}{p-1}} \notag\\
	\le&\, \frac{p-1}{p} \int_{\partial\Omega}|\nabla u|^p\left[\left(\frac{|\Omega|}{\omega_N}\right)^{\frac{1}{N}}(1+\delta_\Omega)^{\frac{1}{p-1}}-\left<x-y,\nu\right>\right] d\mathcal{H}^{N-1}(x) \notag\\
	&+\frac{p-1}{p}\int_{\partial\Omega}|\nabla u|^p\left<x-y,\nu\right>d\mathcal{H}^{N-1}(x), \label{est-eq:integral-U}
\end{align}
for any given $y\in\mathbb{R}^N$. On the other hand, notice that the following Pohozaev identity holds true (see, for instance, \cite{GV89} or \cite[Lemma 4.2]{CL24}):
\begin{equation}\label{est-eq:Pohozaev}
	N\int_\Omega F(u)\,dx+\frac{p-N}{p}\int_{\Omega}uf(u)\,dx= \frac{p-1}{p}\int_{\partial\Omega}|\nabla u|^p\left<x-y,\nu\right>d\mathcal{H}^{N-1}(x),
\end{equation} 
where $\nu$ is the unit outer normal to $\Omega$. Thus, we conclude from \eqref{est-eq:integral-U} and \eqref{est-eq:Pohozaev} that
		\begin{align}
			\mathfrak{D}&\le \frac{p-1}{p} \int_{\partial\Omega}|\nabla u|^p\left[\left(\frac{|\Omega|}{\omega_N}\right)^{\frac{1}{N}}(1+\delta_\Omega)^{\frac{1}{p-1}}-\left<x-y,\nu\right>\right] d\mathcal{H}^{N-1}(x)\notag\\
			&\le  \frac{p-1}{p} \|\nabla u\|_{L^\infty(\partial\Omega)}^p\int_{\partial\Omega}\left|\left(\frac{|\Omega|}{\omega_N}\right)^{\frac{1}{N}}(1+\delta_\Omega)^{\frac{1}{p-1}}-\left<x-y,\nu\right>\right| d\mathcal{H}^{N-1}(x).\label{est-eq:upper}
		\end{align}
	
	Let $T$ be the Brenier map between $\Omega$ and $\left(\frac{|\Omega|}{\omega_N}\right)^{\frac{1}{N}}B_1$ such that $T=\nabla\varphi$ for some convex function $\varphi$ on $\mathbb{R}^N$ and $T\in BV(\mathbb{R}^N;\left(\frac{|\Omega|}{\omega_N}\right)^{\frac{1}{N}}B_1)$ (the existence of such a map can be guaranteed by the Brenier--McCann theorem, as in \cite[Proof of Theorem 2.3]{FMP2010}). We have
	\begin{align}
		&\int_{\partial\Omega}\left|\left(\frac{|\Omega|}{\omega_N}\right)^{\frac{1}{N}}(1+\delta_\Omega)^{\frac{1}{p-1}}-\left<x-y,\nu\right>\right| d\mathcal{H}^{N-1}(x)\notag\\
		\le \,& \int_{\partial\Omega}\left|\left(\frac{|\Omega|}{\omega_N}\right)^{\frac{1}{N}}(1+\delta_\Omega)^{\frac{1}{p-1}}-\left<\mathrm{tr}(T)(x),\nu\right>\right| d\mathcal{H}^{N-1}(x)\notag\\
		&+\int_{\partial\Omega}\left|\mathrm{tr}(T)(x)-x+y\right| d\mathcal{H}^{N-1}(x),\label{est-eq:upper-T}
		\end{align}
	where $\mathrm{tr}(T)$ is the trace of $T$ on $\partial\Omega$.

	As argued in \cite{FMP2010} for formula (2.27) therein (see also \cite[formulas (2.10) and (2.11)]{FI13}), one can find by using the divergence theorem of a form presented in \cite[Formula (2.18)]{FMP2010} that
	$$\int_{\partial\Omega}\left<\mathrm{tr}(T),\nu\right>d\mathcal{H}^{N-1}\ge N|\Omega|.$$
	Also, notice that $|\mathrm{tr}(T)(x)|\le \left(\frac{|\Omega|}{\omega_N}\right)^{\frac{1}{N}}$ for $\mathcal{H}^{N-1}$-a.e. $x\in\partial\Omega$. Thus, we deduce that
	\begin{equation}\label{est-eq:upper-T-1}
		\int_{\partial\Omega}\left|\left(\frac{|\Omega|}{\omega_N}\right)^{\frac{1}{N}}(1+\delta_\Omega)^{\frac{1}{p-1}}-\left<\mathrm{tr}(T)(x),\nu\right>\right| d\mathcal{H}^{N-1}\le N|\Omega|\left[(1+\delta_\Omega)^{\frac{p}{p-1}}-1\right].
	\end{equation}

	Furthermore, by virtue of the Poincar\'e trace inequality \eqref{pre-eq:Poincare-trace} we infer that there is a vector $\mathfrak{c}\in\mathbb{R}^N$ such that
	\begin{equation}\label{est-eq:T-boundary}
		\int_{\partial\Omega}\left|\mathrm{tr}(T)(x)-x+\mathfrak{c}\right| d\mathcal{H}^{N-1}(x)\le NC_\Omega \|DS\|(\Omega)
	\end{equation}
	where $S=T-\mathrm{Id}$ ($\mathrm{Id}$ denotes the identity map). On the other hand, by \cite[Corollary 2.4]{FMP2010} we obtain that
	\begin{equation}\label{est-eq:DS}
			\|DS\|(\Omega)\le 9N^2|\Omega|\sqrt{\delta_\Omega}\, ,
	\end{equation}
	provided $\delta_\Omega\le 1$. 
	
	Hence, by taking $y=\mathfrak{c}$ in \eqref{est-eq:upper-T} and exploiting \eqref{est-eq:upper-T-1}--\eqref{est-eq:DS}, we arrive at 
	\begin{align*}
		&\int_{\partial\Omega}\left|\left(\frac{|\Omega|}{\omega_N}\right)^{\frac{1}{N}}(1+\delta_\Omega)^{\frac{1}{p-1}}-\left<x-a,\nu\right>\right| d\mathcal{H}^{N-1}(x)\\
		&\le \frac{p\,2^{\frac{1}{p-1}}}{p-1}N|\Omega|\delta_\Omega+9N^3C_\Omega|\Omega| \sqrt{\delta_\Omega}\\
		&\le \frac{p}{p-1}(2^{\frac{1}{p-1}}N+9N^3C_\Omega)|\Omega| \sqrt{\delta_\Omega},
	\end{align*}
	provided $\delta_\Omega\le 1$. This, together with \eqref{est-eq:upper}, proves \eqref{est-eq:D-iso}. The proof is complete.
	\end{proof}
	
The following lemma was actually shown in the proof of \cite[Theorem 1]{Ser13} (see formula (2.14) therein). Here we recall it and present its proof since the corresponding arguments will be exploited to prove the results in the forthcoming sections. 
	
	\begin{lemma}\label{est-lem:positivity-D}
	Let $\Omega$, $f$ and $u$ be as in Theorem \ref{thm:main}. Assume either $p\ge N$ or that
	\begin{equation}\label{est-eq:condition-f}
		1<p<N \,\text{  and  }\, \phi\le f\le \frac{Np}{N-p}\phi
	\end{equation}
	for some nonincreasing function $\phi\ge 0$. Then $\mathfrak{D}\ge 0$. 
	\end{lemma}

\begin{proof}
For convenience, denote $\alpha=\frac{p}{p-1}$ and $\beta=\frac{p-N}{N(p-1)}$ as in \cite{Ser13}. Let $$ I_u(t)=\int_{\Lambda_u^t} f(u)\,dx\quad\text{and}\quad K(t)=[I_u(t)]^\alpha[J_u(t)]^\beta.$$ From the proof of \cite[Lemma 4]{Ser13}, it is known that $I'_u(t)=f(t)J'_u(t)$ for a.e. $t\in (0,M)$. Moreover,  it was derived in \cite[Formula (2.12)]{Ser13} that the following isoperimetric-H\"older type inequality holds true\footnote{The inequality \eqref{est-eq:iso-Holder} is valid for any $p>1$ without requiring the extra assumption on $f$ made in \eqref{est-eq:condition-f}. Actually, that assumption is simply used to guarantee the function $K(t)$ to be nonincreasing when $1<p<N$, as explained in \cite[Remark 2]{Ser13}.} 
\begin{equation}\label{est-eq:iso-Holder}
	-I_u^{\alpha-1}(t)J'_u(t)\ge \kappa_N^{\alpha}J_u^{1-\beta}(t)\quad\text{for a.e. }t\in (0,M).
\end{equation}
Since $K(t)$ is nonincreasing on $[0,M]$ according to \cite[Remark 2]{Ser13}, by \eqref{est-eq:iso-Holder} one can deduce that
	\begin{align}
		\left[\int_\Omega f(u)\,dx\right]^\alpha |\Omega|^\beta&\ge\int_0^M -K'(t)\,dt \notag\\
		&=\int_0^M -(\alpha J_u(t) f(t)+\beta I_u(t)) I_u^{\alpha-1}(t)J_u^{\beta-1}(t)J'_u(t)\,dt \label{est-eq:dK}\\
		&\ge\kappa_N^{\alpha}\int_0^M (\alpha J_u(t) f(t)+\beta I_u(t))\,dt \label{est-eq:lower-dK}\\
		&=\kappa_N^{\alpha}\int_0^M \int_{\Omega}(\alpha f(t)+\beta f(u(x)))\chi_{\Omega_t}\,dx\,dt \notag\\
		&=\kappa_N^{\alpha}\int_{\Omega}\int_0^u(\alpha f(t)+\beta f(u(x)))\,dt\,dx \notag\\
		&=\frac{\alpha\kappa_N^{\alpha}}{N}\left(N\int_\Omega F(u)\,dx+\frac{p-N}{p}\int_\Omega uf(u)\,dx\right). \label{est-eq:integral-L}
	\end{align}
	This shows $\mathfrak{D}\ge 0$.
\end{proof}

	\section{A comparison result of Talenti type} \label{sec:com}
	In this section, we study the solution $v$ to problem \eqref{in-eq:pro-v}. In the spirit of Talenti \cite{Talenti76,Talenti79}, we develop a comparison result for $v$ in terms of solutions $u$ to problem \eqref{in-eq:pro}; see Theorem \ref{com-thm:Talenti} below. 
	
	
	\begin{theorem} \label{com-thm:Talenti}
		Let $p>1$ and let $\Omega$, $f$ and $u$ be as in Theorem \ref{thm:main}. Let $v$ be the solution to \eqref{in-eq:pro-v}. Then $v(x)\geq u^*(x)$ for $x\in\Omega^*$. 
		
	Furthermore, under the assumptions of Theorem \ref{thm:main}, there exists a constant $C$, depending only on $N,p,M_f,m_f$ (and on $s$ if $1<p<N$), such that
\begin{equation}\label{com-eq:bound-v-u*}
	\|v-u^*\|_{L^\infty(\Omega^*)}\le C\mathfrak{D}^{\frac{p}{N(p-1)+p}},
\end{equation}
where $\mathfrak{D}$ is given by \eqref{in-eq:def-D}.
\end{theorem}

\begin{remark}
	We mention that problem \eqref{in-eq:pro-v} was considered previously by Kesavan and Pacella \cite{KP94} in the study of a generalized form of the Payne--Rayner inequality. Under the assumption that $f$ is continuous and satisfies $f(t)>0$ for $t>0$, they demonstrated there that the quality case of this inequality forces $v=u^*=u$, by exploiting the result in \cite{BZ} on the characterization of extremals in the P\'olya--Szeg\"o principle. Based on this, in \cite{KP94} they also established symmetry and monotonicity property for solutions to problem \eqref{in-eq:pro} when $\Omega$ is a ball and $p=N$, extending the work of Lions \cite{Lions81}.
\end{remark}

The proof of Theorem \ref{com-thm:Talenti} exploits the arguments deriving $\mathfrak{D}\ge 0$ in the proof of Lemma \ref{est-lem:positivity-D}. To begin with, let us define
\begin{equation}\label{com-eq:def-t-et}
	t_{\epsilon,\tau}=\sup \left\{t>0:J_u(t)>\epsilon^{\frac{N\tau}{N-1}}\right\},
\end{equation} 
for $\epsilon,\tau>0$ such that $\epsilon^{\frac{N\tau}{N-1}}<|\Omega|$. Since $J_u(t)$ is right continuous, we have
\begin{equation}\label{com-eq:Jt_et}
	J_u(t_{\epsilon,\tau})\le \epsilon^{\frac{N\tau}{N-1}}.
\end{equation}
We also notice that the solution $v$ to \eqref{in-eq:pro-v} is radially symmetric and it is given by
\begin{equation}\label{com-eq:formula-v}
	v(x)=\tilde{v}(\omega_N|x|^N)=(\kappa_N)^{-\frac{p}{p-1}}\int_{\omega_N|x|^N}^{|\Omega|}\xi^{\frac{(1-N)p}{N(p-1)}}\left[\int_0^\xi f(u^\#(s))\,ds\right]^{\frac{1}{p-1}}d\xi.
\end{equation}

\begin{proof}[Proof of Theorem \ref{com-thm:Talenti}]
Since $u\in W_0^{1,p}(\Omega)$, it follows from (i) in Lemma \ref{pre-lem:prop-rea} that $u^\#(J_u(t))=t $ for any $t\in [0,M]$. Thus, $$I'_u(t)=f(t)J'_u(t)=f(u^\#(J_u(t)))J'_u(t)\quad\text{for a.e. }t\in (0,M),$$ which entails 
\begin{equation}\label{com-eq:formula-I}
	I_u(t)=\int_0^{J_u(t)}f(u^\#(s))\,ds.
\end{equation}
Then, recall from \eqref{est-eq:iso-Holder} that
\begin{equation}\label{com-eq:iso-Holder-v}
	1\le -\kappa_N^{-\alpha}I_u^{\alpha-1}(t)J_u^{\beta-1}(t)J'_u(t) \quad\text{for a.e. }t\in (0,M),
\end{equation}
where $\alpha=\frac{p}{p-1}$ and $\beta=\frac{p-N}{N(p-1)}$. Thus, given $0<t\le M$, we integrate \eqref{com-eq:iso-Holder-v} over $(0,t)$ and use \eqref{com-eq:formula-I} to get
\begin{align}
	u^\#(J_u(t))=t& \le \kappa_N^{-\alpha}\int_0^t -I_u^{\alpha-1}(t)J_u^{\beta-1}(t)J'_u(t) \,dt \label{com-eq:compare-u-v} \\
	&=\kappa_N^{-\alpha}\int_{J_u(t)}^{|\Omega|}\xi^{\beta-1}\left[\int_0^\xi f(u^\#(s))\,ds\right]^{\alpha-1}d\xi \notag\\
	&= \tilde{v}(J_u(t)), \notag
\end{align}
where $\tilde{v}$ is given in \eqref{com-eq:formula-v}. This implies $u^*(x)\le v(x)$ for $x\in\Omega^*$.

To show \eqref{com-eq:bound-v-u*}, we first observe from \eqref{com-eq:compare-u-v} that
\begin{equation}\label{com-eq:formula-v-u*}
	\|v-u^*\|_{L^\infty(\Omega^*)}=\kappa_N^{-\alpha}\int_0^M \left[ -I_u^{\alpha-1}(t)J_u^{\beta-1}(t)J'_u(t)-\kappa_N^\alpha\right]dt.
 \end{equation}
 On the other hand, in view of \eqref{est-eq:dK} and \eqref{est-eq:lower-dK}, we find that
\begin{equation}\label{com-eq:ajfbi-D}
	\kappa_N^{-\alpha}\int_0^M(\alpha J_u(t) f(t)+\beta I_u(t))\left[ -I_u^{\alpha-1}(t)J_u^{\beta-1}(t)J'_u(t)-\kappa_N^\alpha\right]dt\le \frac{N}{\alpha}\mathfrak{D}.
\end{equation}

We aim to conclude \eqref{com-eq:bound-v-u*} by combining \eqref{com-eq:formula-v-u*} and \eqref{com-eq:ajfbi-D}, under the assumptions of Theorem \ref{thm:main}. For this, we first claim that, there is a constant $\sigma=\sigma(p)$ such that
\begin{equation}\label{com-eq:low-ajfbi}
	\alpha J_u(t) f(t)+\beta I_u(t)\ge \sigma m_f J_u(t)\quad\text{for any }t\in (0,M).
\end{equation}
Indeed, since $\beta\ge0$ when $p\ge N$, it holds in this case 
$$\alpha J_u(t) f(t)+\beta I_u(t)\ge (\alpha+\beta) m_f J_u(t).$$ 
When $p<N$, by assumption \eqref{in-eq:conditions-f} we infer that
\begin{align*}
I_u(t)=\int_{\Lambda_u^t} f(u)\,dx\le\frac{Np-s}{N-p}\int_{\Lambda_u^t}\phi(u)\,dx & \le \frac{Np-s}{N-p}\phi(t)J_u(t) \\
& \le \frac{Np-s}{N-p}f(t)J_u(t),
\end{align*}
thus leading to
\begin{equation*}
\alpha J_u(t) f(t)+\beta I_u(t)\ge \frac{s\,m_f}{N(p-1)}J_u(t).
\end{equation*}
Hence, \eqref{com-eq:low-ajfbi} holds with 
\begin{equation}\label{com-eq:def-sigma}
	\sigma=
	\begin{cases}
		1+\frac{p}{N(p-1)} & \text{if }p\ge N\\
		\frac{s}{N(p-1)} & \text{if }1<p<N.
		\end{cases}
\end{equation}

Exploiting \eqref{com-eq:def-t-et}, \eqref{com-eq:Jt_et}, \eqref{com-eq:iso-Holder-v} and \eqref{com-eq:low-ajfbi}, we obtain that, for any $\epsilon,\tau>0$ with $\epsilon^{\frac{N\tau}{N-1}}<|\Omega|$, 
\begin{align*}
	&\int_0^M \left[ -I_u^{\alpha-1}(t)J_u^{\beta-1}(t)J'_u(t)-\kappa_N^\alpha\right]dt\\
	=& \int_0^{t_{\epsilon,\tau}} \left[ -I_u^{\alpha-1}(t)J_u^{\beta-1}(t)J'_u(t)-\kappa_N^\alpha\right]dt+ \int_{t_{\epsilon,\tau}}^M \left[ -I_u^{\alpha-1}(t)J_u^{\beta-1}(t)J'_u(t)-\kappa_N^\alpha\right]dt\\
	\le & \int_0^{t_{\epsilon,\tau}}\frac{\alpha J_u(t) f(t)+\beta I_u(t)}{\sigma m_f J_u(t)}\left[ -I_u^{\alpha-1}(t)J_u^{\beta-1}(t)J'_u(t)-\kappa_N^\alpha\right]dt \\ 
	& + (M_f)^{\alpha-1} \int_{t_{\epsilon,\tau}}^M -J_u^{\alpha+\beta-2}(t)J'_u(t)\,dt\\
	\le & \,\frac{\epsilon^{\frac{-N\tau}{N-1}}}{\sigma m_f}\int_0^M (\alpha J_u(t) f(t)+\beta I_u(t))\left[ -I_u^{\alpha-1}(t)J_u^{\beta-1}(t)J'_u(t)-\kappa_N^\alpha\right]dt\\
	& + \frac{(M_f)^{\alpha-1}}{\alpha+\beta-1} J_u^{\alpha+\beta-1}(t_{\epsilon,\tau})\\
	\le & \, \frac{N\kappa_N^{\alpha}}{\alpha \sigma m_f}\epsilon^{\frac{-N\tau}{N-1}}\mathfrak{D}+\frac{(M_f)^{\alpha-1}}{\alpha+\beta-1}\epsilon^{\frac{N\tau(\alpha+\beta-1)}{N-1}}.
\end{align*}
Thus, in terms of \eqref{com-eq:formula-v-u*}, we conclude
\begin{equation*}
		\|v-u^*\|_{L^\infty(\Omega^*)}\le \frac{N(p-1)}{p \sigma m_f}\epsilon^{\frac{-N\tau}{N-1}}\mathfrak{D}+\frac{N(p-1)(M_f)^{\frac{1}{p-1}}}{p \kappa_N^{\frac{p}{p-1}}}\epsilon^{\frac{p\tau}{(N-1)(p-1)}}.
\end{equation*}
This validates \eqref{com-eq:bound-v-u*} by taking $\epsilon=\mathfrak{D}$ and $\tau=\frac{(N-1)(p-1)}{N(p-1)+p}$, provided 
\begin{equation}\label{eq:DleO}
	\mathfrak{D}^{\frac{N(p-1)}{N(p-1)+p}}<|\Omega|.
\end{equation}
If \eqref{eq:DleO} is false, observe from \eqref{com-eq:formula-v} that
\begin{align*}
		\|v-u^*\|_{L^\infty(\Omega^*)}\le \max_{\Omega^*} v=v(0)& \le \frac{N(p-1)(M_f)^{\frac{1}{p-1}}}{p \kappa_N^{\frac{p}{p-1}}} |\Omega|^{\frac{p}{N(p-1)}}\\
		& \le \frac{N(p-1)(M_f)^{\frac{1}{p-1}}}{p \kappa_N^{\frac{p}{p-1}}} \mathfrak{D}^{\frac{p}{N(p-1)+p}},
\end{align*}
still verifying \eqref{com-eq:bound-v-u*}. 

The proof is complete.
\end{proof}
	
	\begin{remark}\label{com-rk:C-s-mf}
		As seen in the above proof, the constant $C$ in \eqref{com-eq:bound-v-u*} is actually derived explicitly and it is given by
		$$C=\frac{N(p-1)}{p \sigma m_f}+\frac{N(p-1)(M_f)^{\frac{1}{p-1}}}{p \kappa_N^{\frac{p}{p-1}}},$$
where $\sigma$ is given in \eqref{com-eq:def-sigma}.
	\end{remark}

\section{Proof of Theorem \ref{thm:main}} \label{sec:pf}

 To prove Theorem \ref{thm:main}, in this section we first apply the quantitative P\'olya--Szeg\"o principle presented in Theorem \ref{pre-thm:PZ} to obtain a preliminary estimate on the $L^1$ distance between the solution $u$ to \eqref{in-eq:pro} and its Schwarz symmetrization $u^*$; see Proposition \ref{pf-pro:u-ustar} below. This, together with Proposition \ref{est-pro:D-iso} and Theorem \ref{com-thm:Talenti}, concludes the result in Theorem \ref{thm:main}, as outlined in Subsection \ref{sub-sec:strategy}.

Hereafter, as agreed for $u$, we regard $v$ as a function in $\mathbb{R}^N$ by extending it to be 0 outside $\overline{\Omega^*}$.


\begin{proposition} \label{pf-pro:u-ustar}
	Let $\Omega$, $f$ and $u$ be as in Theorem \ref{thm:main}. Let $v$ be the solution to \eqref{in-eq:pro-v}. Suppose that the assumptions stated in Theorem \ref{thm:main} on $p$ and $f$ hold. Then there exist constants $\theta=\theta(N,p)$ and $C=C(N,p,|\Omega|,M,M_f,m_f)$ such that
	\begin{equation}\label{pf-eq:L1-u-ustar}
		\inf_{x_0\in\mathbb{R}^N}\int_{\mathbb{R}^N}|u(x)- u^*(x+x_0)|\,dx\le C \|v-u^*\|_{L^\infty(\Omega^*)}^{\theta}.
	\end{equation}
\end{proposition}

The proof of Proposition \ref{pf-pro:u-ustar} is based on Theorem \ref{pre-thm:PZ}. The key ingredient is to estimate the quantities $\mathcal{E}_u$ and $\mathcal{M}_{u^*}(\delta)$ appearing in \eqref{pre-eq:quant-PS} in terms of the size of $\|v-u^*\|_{L^\infty(\Omega^*)}$. This is achieved in the following two lemmas by adapting the arguments exploited in \cite{ABMP23} for deriving a quantitative Talenti's inequality.

\begin{lemma} \label{pf-lem:Du-Dustar}
	Under the assumptions of Proposition \ref{pf-pro:u-ustar}, it holds
	\begin{equation*}
	\int_{\mathbb{R}^N} |\nabla u|^p\,dx-\int_{\mathbb{R}^N} |\nabla u^*|^p\,dx\le pM_f|\Omega| \|v-u^*\|_{L^\infty(\Omega^*)}.
	\end{equation*}
	\end{lemma}

\begin{proof}
	 By the weak formulation of \eqref{in-eq:pro} and that of \eqref{in-eq:pro-v}, we infer from Theorem \ref{com-thm:Talenti} that
	\begin{equation}\label{eq:Du-Dv}
		\int_{\mathbb{R}^N} |\nabla u|^p\,dx=\int_{\mathbb{R}^N} u f(u)\,dx=\int_{\mathbb{R}^N} u^*f(u^*)\,dx\le \int_{\mathbb{R}^N} v f(u^*)\,dx=\int_{\mathbb{R}^N} |\nabla v|^p\,dx,
	\end{equation}
where, in the second equality, we have used that the function $tf(t)$ is a nonnegative Borel function so that it preserves integrals under rearrangements (see, for instance, \cite[Theorem 1.1.1]{Kes06}).
	
	On the other hand, notice that $v$ minimizes the functional
	$$\frac{1}{p}\int_{\Omega} |\nabla\psi|^p\,dx-\int_{\Omega} f(u^*)\psi\,dx$$
	among all $\psi\in W_0^{1,p}(\Omega^*)$, which implies
	\begin{align}
	\int_{\mathbb{R}^N} |\nabla v|^p\,dx-\int_{\mathbb{R}^N} |\nabla u^*|^p\,dx & \le p \int_{\mathbb{R}^N} f(u^*)(v-u^*)\,dx \notag\\
	& \le pM_f|\Omega|\|v-u^*\|_{L^\infty(\Omega^*)}. \label{pf-eq:Dv-Dustar}
	\end{align}
	This together with \eqref{eq:Du-Dv} proves the lemma.
	\end{proof}


	
	\begin{lemma} \label{pf-lem:mea-Dustar}
	Let $\delta>0$. Under the assumptions of Proposition \ref{pf-pro:u-ustar}, there exists a constant $C=C(N,p,|\Omega|,M_f,m_f)$ such that
		\begin{equation}\label{pf-eq:mea-Dustar}
				\left|\left\{x\in \Omega^*:|\nabla u^*(x)|\le	\delta\right\}\right| \le C\left[\delta^{N(p-1)}+\left(\|v-u^*\|_{L^\infty(\Omega^*)}\right)^{\frac{N(p-1)}{2N(p-1)+1}}\right].
		\end{equation}
		\end{lemma}
	
	
	\begin{proof}
		For simplicity, set
		$$A:=\{x\in \Omega^*:|\nabla u^*(x)|\le\delta\} \quad\text{and}\quad \varepsilon:=\|v-u^*\|_{L^\infty(\Omega^*)}.$$ 
		Let $\tau=\frac{(N-1)(p-1)}{2N(p-1)+1}$. It suffices to prove the assertion when $0<\varepsilon<\min\{1, |\Omega|^{\frac{N-1}{N\tau}}\}$, since \eqref{pf-eq:mea-Dustar} is automatically valid if $\varepsilon$ is large. We shall validate \eqref{pf-eq:mea-Dustar} by estimating the measure of the following three subsets of $A$ respectively:
		\begin{gather*}
			A_1:=A \cap \left\{x\in \Omega^*:u^*(x)\in (0,t_{\varepsilon,\tau})\setminus S\right\}\\
			A_2:=A\cap \left\{x\in \Omega^*:u^*(x)\in (0,t_{\varepsilon,\tau})\cap S\right\} \\
			A_3:=A \cap \left\{x\in \Omega^*:u^*(x)\ge t_{\varepsilon,\tau}\right\}.
		\end{gather*}
		Let $t_{\varepsilon,\tau}$ be given by \eqref{com-eq:def-t-et} and set $$S:=\left\{t\in(0,M): \int_{\Gamma_v^t} |\nabla v|^{p-1}\,d\mathcal{H}^{N-1}-\int_{\Gamma_{u^*}^t} |\nabla u^*|^{p-1}\,d\mathcal{H}^{N-1}>\varepsilon^{\frac{N\tau}{N-1}}\right\}.$$
		

	Clearly, by \eqref{com-eq:Jt_et}, we have 
	\begin{equation}\label{pf-eq:est-A3}
		|A_3|\le |\Lambda^{t_{\varepsilon,\tau}}_{u^*}|=|\Lambda^{t_{\varepsilon,\tau}}_{u}|=J_u(t_{\varepsilon,\tau})\le \varepsilon^{\frac{N\tau}{N-1}}.
	\end{equation}
	Notice that here we have used the fact that $|\Gamma^{t_{\varepsilon,\tau}}_u|=0$.\footnote{Actually, by exploiting the regularity result \cite[Theorem 1.1]{Lou08} (see also \cite{ACF2021}), we infer that $|Z_u\cap\Omega|=0$ holds for a positive solution $u$ to problem \eqref{in-eq:pro} with $f$ positive. This implies that all level sets $\Gamma_u^t$ of $u$ fulfill $|\Gamma_u^t|=0$ for any $t>0$. Meanwhile, in this case we infer from (i) and (iii) in Lemma \ref{pre-lem:prop-rea} that $J_u(t)$ is absolutely continuous on $(0,M)$ and $|Z_{u^*}\cap\Omega^*|=0$.}

	Next, we estimate $|A_1|$ and $|A_2|$ by noting that $|\nabla v|$ is explicitly known. Indeed, one can calculate from \eqref{com-eq:formula-v} that
	\begin{equation}\label{pf-eq:formula-Dv}
		|\nabla v(x)|=\left(\frac{|x|}{N}\right)^{\frac{1}{p-1}}\left[\frac{1}{\omega_N|x|^N}\int_0^{\omega_N|x|^N}f(u^{\#}(s))\,ds\right]^{\frac{1}{p-1}}
	\end{equation}
	for any $x\neq 0$ ($\nabla v(0)=0$).

	As explained in the last footnote, when $m_f>0$ we know that $J_u(t)$ is absolutely continuous. Then by (i) in Lemma \ref{pre-lem:prop-rea} we can observe that
	\begin{equation}\label{pf-eq:levelset-v-ustar}
		\Gamma_v^t=\partial B_{\left[\frac{J_v(t)}{\omega_N}\right]^{\frac1N}}=\partial\Lambda_v^t \,,\quad \Gamma_{u^*}^t=\partial B_{\left[\frac{J_u(t)}{\omega_N}\right]^{\frac1N}}=\partial\Lambda_{u^*}^t \,,
	\end{equation}
	for each $t>0$. Thus, for $t\in(0,t_{\varepsilon,\tau})\setminus S$, $x\in \Gamma_{u^*}^t$ and $y\in \Gamma_v^t$, we have
	\begin{align}
		&|\nabla v(y)|^{p-1}-|\nabla u^*(x)|^{p-1} \notag\\
		\le & \, \frac{1}{\mathcal{H}^{N-1}(\Gamma_{u^*}^t)}\left( \int_{\Gamma_v^t} |\nabla v|^{p-1}\,d\mathcal{H}^{N-1}-\int_{\Gamma_{u^*}^t} |\nabla u^*|^{p-1}\,d\mathcal{H}^{N-1}\right) \notag\\
		\le &\, \frac{\varepsilon^{\frac{N\tau}{N-1}}}{\kappa_N J_u^{\frac{N-1}{N}}(t)}\le \frac{\varepsilon^{\frac{\tau}{N-1}}}{\kappa_N}.\label{pf-eq:D-v-ustar}
	\end{align} 
Here we have used the fact that $u^*$ and $v$ are both radial, and that $J_v(t)\ge J_u(t)$ for any $t\ge0$, since $v\ge u^*$ in $\Omega^*$. Also, from \eqref{pf-eq:levelset-v-ustar} it follows that $|\nabla v(y)|=|\nabla v(l(t)x)|$ where $l(t)=\left[\frac{J_v(t)}{J_u(t)}\right]^{\frac1N}$. In view of \eqref{pf-eq:formula-Dv}, we notice that
\begin{equation}\label{pf-eq:low-Dv}
	|\nabla v(x)|\ge \left(\frac{m_f|x|}{N}\right)^{\frac{1}{p-1}}\quad\forall\,x\in \Omega^*.
\end{equation}
Hence, for $x\in A_1$, we derive from \eqref{pf-eq:D-v-ustar} and \eqref{pf-eq:low-Dv} that
	$$\frac{m_f}{N}|x|\le |\nabla v(l(t)x)|^{p-1} \le |\nabla u^*(x)|^{p-1}+\frac{\varepsilon^{\frac{\tau}{N-1}}}{\kappa_N} \le \delta^{p-1}+\frac{\varepsilon^{\frac{N\tau}{N-1}}}{\kappa_N}.$$
	This implies 
	\begin{equation}\label{pf-eq:est-A1}
		|A_1|\le C_1[\delta^{N(p-1)}+\epsilon^{\frac{N\tau}{N-1}}]
	\end{equation}
	for some constant $C_1=C_1(N, m_f)$.

	 To estimate $|A_2|$, we first estimate the size of the singular set $S$. For this, let $$I_v(t)=\int_{\Lambda^t_v} f(u^*(x))\,dx=\int_0^{J_v(t)} f(u^{\#}(s))\,ds.$$
	By \eqref{com-eq:formula-I}, $I_v(t)\ge I_u(t)$ for any $t\ge 0$. Since $|Z_v\cap\Omega^*|=0$ by \eqref{pf-eq:low-Dv}, we deduce from \eqref{pre-eq:formula-Ju} and \eqref{pf-eq:levelset-v-ustar} that $J'_v(t)=-\frac{\kappa_N J_v^{\frac{N-1}{N}}(t)}{|\nabla v(z)|}$ for $t>0$, where $z\in\Gamma_v^t$. By applying the divergence theorem to equation \eqref{in-eq:pro-v} in $\Lambda_v^t$, we find $$I_v(t)=\kappa_N J_v^{\frac{N-1}{N}}(t) |\nabla v(z)|^{p-1}\quad\text{for }t>0.$$
	Thus, recalling \eqref{est-eq:iso-Holder}, we obtain
		\begin{equation}\label{pf-eq:IJdJ}
			-I_v^{\alpha-1}(t)J_v^{\beta-1}(t)J'_v(t)= \kappa_N^{\alpha}\le 	-I_u^{\alpha-1}(t)J_u^{\beta-1}(t)J'_u(t)
		\end{equation}
	for a.e. $t\in (0,M)$; here, $\alpha=\frac{p}{p-1}$ and $\beta=\frac{p-N}{N(p-1)}$ as in the proof of Lemma \ref{est-lem:positivity-D}. From this, we get
	\begin{equation}\label{pf-eq:dJv-dJu}
		-J'_v(t)\le -J'_u(t)\quad\text{for a.e. } t\in (0,M).
	\end{equation}

	In terms of \eqref{pre-eq:dJu}, \eqref{pf-eq:levelset-v-ustar} and \eqref{pf-eq:dJv-dJu}, we now derive that, for a.e. $t\in (0,M)$,
		\begin{align*}
	\int_{\Gamma_{u^*}^t} |\nabla u^*|^{p-1}\,d\mathcal{H}^{N-1}&=|\nabla u^*(x)|^{p-1}\mathcal{H}^{N-1}(\Gamma_{u^*}^t)\\
	&= (-J'_u(t))^{1-p}[\mathcal{H}^{N-1}(\Gamma_{u^*}^t)]^p\\
	&=(-J'_u(t))^{1-p} \left[\kappa_N J_u^{\frac{N-1}{N}}(t)\right]^p\\
	& \le (-J_v'(t))^{1-p} \left[\kappa_N J_v^{\frac{N-1}{N}}(t)\right]^p\\
	&= (-J_v'(t))^{1-p}[\mathcal{H}^{N-1}(\Gamma_v^t)]^p\\
	&=|\nabla v(y)|^{p-1}\mathcal{H}^{N-1}(\Gamma_v^t)\\
	&=\int_{\Gamma_v^t} |\nabla v|^{p-1}\,d\mathcal{H}^{N-1},
		\end{align*}
	where $x\in\Gamma_{u^*}^t$ and $y\in\Gamma_v^t$. Thus, by \eqref{pf-eq:Dv-Dustar} and the co-area formula (see, for instance, \cite[Proposition 2.1]{BZ}), we obtain
		\begin{align*}
			p M_f|\Omega|\varepsilon &\ge \int_{\mathbb{R}^N} |\nabla v|^p\,dx-\int_{\mathbb{R}^N} |\nabla u^*|^p\,dx\\ & \ge \int_0^M \left(\int_{\Gamma_v^t} |\nabla v|^{p-1}\,d\mathcal{H}^{N-1}-\int_{\Gamma_{u^*}^t} |\nabla u^*|^{p-1}\,d\mathcal{H}^{N-1}\right)dt\\
			& \ge \int_{S}\left(\int_{\Gamma_v^t} |\nabla v|^{p-1}\,d\mathcal{H}^{N-1}-\int_{\Gamma_{u^*}^t} |\nabla u^*|^{p-1}\,d\mathcal{H}^{N-1}\right)dt\\
			& \ge |S|\varepsilon^{\frac{N\tau}{N-1}}.
		\end{align*}
		This yields 
		\begin{equation}\label{pf-eq:est-S}
				|S|\le p M_f|\Omega| \varepsilon^{1-\frac{N\tau}{N-1}}.
		\end{equation}

		With \eqref{pf-eq:est-S}, we are able to further derive estimates for the measure of the sets 
		\begin{equation*}
			\{x\in \Omega^*:v(x)\in (0,t_{\varepsilon,\tau})\setminus S\}\quad \text{and}\quad \{x\in \Omega^*:v(x)\in S\},
		\end{equation*}
		which helps us to conclude an estimate for $|A_2|$. To this aim, consider the set
		$$A_v(\mu):=\{x\in\Omega^*:|\nabla v(x)|\le\mu\},$$ 
		where $\mu>0$. By virtue of  \eqref{pf-eq:low-Dv}, we see
		$$|A_v(\mu)|\le \omega_N\left(\frac{N\mu^{p-1}}{m_f}\right)^N.$$ 
		Then, by the co-area formula and \eqref{pf-eq:est-S}, we can estimate as follows: 
		\begin{align}
			|\{x\in \Omega^*:v(x)\in & S\}|  =\,\int_{S}\int_{\Gamma_v^t}\frac{1}{|\nabla v|}\,d\mathcal{H}^{N-1}\,dt \notag\\
			=\,&\int_{S}\int_{\Gamma_v^t\cap A_v(\mu)}\frac{1}{|\nabla v|}\,d\mathcal{H}^{N-1}\,dt+\int_{S}\int_{\Gamma_v^t\setminus A_v(\mu)}\frac{1}{|\nabla v|}\,d\mathcal{H}^{N-1}\,dt \notag\\
			\le\, & |A_v(\mu)|+\frac{1}{\mu}\int_S \kappa_N J_v^{\frac{N-1}{N}}(t)\,dt \notag\\
			\le\, & |A_v(\mu)|+\frac{\kappa_N|S|}{\mu}|\Omega|^{\frac{N-1}{N}} \notag\\
			\le\, & \omega_N\left(\frac{N\mu^{p-1}}{m_f}\right)^N + \frac{\kappa_N p M_f|\Omega|^{\frac{2N-1}{N}}}{\mu} \varepsilon^{1-\frac{N\tau}{N-1}}. \label{pf-eq:est-v-S}
		\end{align}
		
		
		Since $v-\varepsilon\le u^*$ in $\Omega^*$, we easily get $$J_v(t+\varepsilon)\le J_u(t)\le J_v(t)$$
	for $0\le t\le M$. Moreover, observe from \eqref{pf-eq:IJdJ} that
	\begin{align*}
	J_v(t+\varepsilon)-J_v(t)=\int_t^{t+\varepsilon}J'_v(t)\,dt= &\int_t^{t+\varepsilon}-\kappa_N^\alpha I_v^{1-\alpha}(t)J_v^{1-\beta}(t)\,dt\\
	\ge & \int_t^{t+\varepsilon}-\kappa_N^\alpha (m_f)^{1-\alpha}J_v^{2-\alpha-\beta}(t)\,dt\\
	\ge & -\kappa_N^{\frac{p}{p-1}} (m_f)^{\frac{-1}{p-1}} |\Omega|^{1-\frac{p}{N(p-1)}}\varepsilon.
	\end{align*}
Hence, setting $C_2=\kappa_N^{\frac{p}{p-1}} (m_f)^{\frac{-1}{p-1}} |\Omega|^{1-\frac{p}{N(p-1)}}$, we obtain
\begin{equation}\label{pf-eq:JugeJv}
	J_u(t)\ge J_v(t)-C_2\varepsilon,\quad\text{for } 0\le t\le M.
\end{equation}
Thus, by the co-area formula we find
\begin{align}
	|\{x\in \Omega^*:u^*(x)\in (0,t_{\varepsilon,\tau})\setminus S\}|&= \int_{(0,t_{\varepsilon,\tau})\setminus S} \int_{\Gamma_{u^*}^t}\frac{1}{|\nabla u^*|}\,d\mathcal{H}^{N-1}\,dt \notag\\
	&= -\int_{(0,t_{\varepsilon,\tau})\setminus S}J'_u(t)\,dt \notag\\
	& \ge  -\int_{(0,t_{\varepsilon,\tau})\setminus S}J'_v(t)\,dt \notag\\
	& = \int_{(0,t_{\varepsilon,\tau})\setminus S} \int_{\Gamma_{u^*}^t}\frac{1}{|\nabla v|}\,d\mathcal{H}^{N-1}\,dt \notag\\
	& = |\{x\in \Omega^*:v(x)\in (0,t_{\varepsilon,\tau})\setminus S\}. \label{pf-eq:ustar-S}
\end{align}

Exploiting \eqref{pf-eq:est-v-S}, \eqref{pf-eq:JugeJv} and \eqref{pf-eq:ustar-S}, we deduce that
\begin{align}
	|A_2| & \le |\{x\in \Omega^*:u^*(x)\in (0,t_{\varepsilon,\tau})\cap S\}| \notag \\
	&= |\Omega|-| \Lambda^{t_{\varepsilon,\tau}}_{u^*}|- |\{x\in \Omega^*:u^*(x)\in (0,t_{\varepsilon,\tau})\setminus S\}| \notag \\
	&\le |\Omega|-J_v(t_{\varepsilon,\tau})+ C_2\varepsilon-|\{x\in \Omega^*:v(x)\in (0,t_{\varepsilon,\tau})\setminus S\}| \notag \\
	& = |\{x\in \Omega^*: v(x)\in (0,t_{\varepsilon,\tau})\cap S\}| +C_2\varepsilon \notag \\
	&\le |\{x\in \Omega^*:v(x)\in S\}|+C_2\varepsilon \notag \\
	&\le  \omega_N\left(\frac{N\mu^{p-1}}{m_f}\right)^N + \frac{\kappa_N p M_f|\Omega|^{\frac{2N-1}{N}}}{\mu} \varepsilon^{1-\frac{N\tau}{N-1}}+C_2\varepsilon. \label{pf-eq:est-A2}
\end{align}
Now, take $\mu=\varepsilon^{\frac{1}{2N(p-1)+1}}$, so that $$\mu^{N(p-1)}=\frac{\varepsilon^{1-\frac{N\tau}{N-1}}}{\mu}=\varepsilon^{\frac{N\tau}{N-1}}.$$
With this choice, we combine \eqref{pf-eq:est-A3}, \eqref{pf-eq:est-A1} and \eqref{pf-eq:est-A2} to conclude that
\begin{equation*}
|A| =|A_1|+|A_2|+|A_3| \le C (\delta^{N(p-1)}+\varepsilon^{\frac{N\tau}{N-1}})
\end{equation*}
for some constant $C=C(N,p,|\Omega|,M_f,m_f)$.

This completes the proof by recalling the definition of $\varepsilon$ and of $\tau$. 
		\end{proof}
	
With Lemmas \ref{pf-lem:Du-Dustar} and \ref{pf-lem:mea-Dustar}, we prove Proposition \ref{pf-pro:u-ustar} in the following.

	\begin{proof}[Proof of Proposition \ref{pf-pro:u-ustar}]
		As in the proof of Lemma \ref{pf-lem:mea-Dustar}, we set $\varepsilon:=\|v-u^*\|_{L^\infty(\Omega^*)}$. We shall prove \eqref{pf-eq:L1-u-ustar} when $0<\varepsilon\le 1$, since if $\varepsilon>1$, \eqref{pf-eq:L1-u-ustar} is clearly valid by noting that
		\begin{equation} \label{pf-eq:2L1u}
			\inf_{x_0\in\mathbb{R}^N}\int_{\mathbb{R}^N}|u(x)-u^*(x+x_0)|\,dx
			\le 2\|u\|_{L^1(\mathbb{R}^N)}\le M|\Omega| \le M|\Omega|\varepsilon^{\theta'}
		\end{equation}
 for any $\theta'>0$.

		Notice that $J_u(0)=|\Omega|$. Given $0<r\le\min\{r_1,r_2\}$ with $r_1, r_2$ as in Theorem \ref{pre-thm:PZ}, we conclude from Lemmas \ref{pf-lem:Du-Dustar} and \ref{pf-lem:mea-Dustar} that 
		\begin{gather*}
		(\mathcal{E}_u)^r\le (p M_f |\Omega|)^r \|\nabla u^*\|_{L^p(\mathbb{R}^N)}^{-pr} \varepsilon^r,\\
		\mathcal{M}_{u^*}(\delta)\le C (\delta^{N(p-1)}+\varepsilon^{\frac{N(p-1)}{2N(p-1)+1}}),
		\end{gather*}
		for some constant $C=C(N,p,|\Omega|,M,M_f,m_f)$, where $\mathcal{E}_u$ and $\mathcal{M}_{u^*}(\delta)$ are defined by \eqref{pre-eq:def-E} and \eqref{pre-eq:def-M}, respectively. 
		Then by taking $\delta=\epsilon^\lambda$ for some $\lambda>0$ to be fixed and by applying Theorem \ref{pre-thm:PZ}, we obtain
		\begin{align}
			&\inf_{x_0\in\mathbb{R}^N}\int_{\mathbb{R}^N}|u(x)-u^*(x+x_0)|\,dx \notag\\
			\lesssim & \left(\varepsilon^{\lambda N(p-1)}+\varepsilon^{\frac{N(p-1)}{2N(p-1)+1}}\right)^{r_3}+ \|\nabla u^*\|_{L^p(\mathbb{R}^N)}^{1-prr_3}\,\varepsilon^{rr_3}+ \|\nabla u^*\|_{L^p(\mathbb{R}^N)}^{1-prr_3+r_3}\,\varepsilon^{r_3(r-\lambda)}, \label{pf-eq:L1-u-ustar-pre}
		\end{align}
	where $r_3$ is as in Theorem \ref{pre-thm:PZ}. Note that 
	\begin{equation}\label{pf-eq:est-Dustar}
		 \int_{\mathbb{R}^N} |\nabla u^*|^p\,dx\le \int_{\mathbb{R}^N} |\nabla u|^p\,dx=\int_{\Omega}uf(u)\,dx\le MM_f|\Omega|,
	\end{equation}
	which follows by the P\'olya--Szeg\"o principle and by testing equation \eqref{in-eq:pro} with $u$. Accordingly, fixing constants $\lambda$ and $r$ such that $\lambda< r <\min\{r_1,r_2, \frac{1}{pr_3}\}$, we deduce from \eqref{pf-eq:L1-u-ustar-pre} that \eqref{pf-eq:L1-u-ustar} holds with $$\theta=\min\left\{\lambda N(p-1) r_3,\, \frac{N(p-1)r_3}{2N(p-1)+1},\, rr_3,\, r_3(r-\lambda)\right\}.$$
	
	The proof is complete.
	\end{proof}
	
Now, we are in the position to present the proof of Theorem \ref{thm:main}.
	
	\begin{proof}[Proof of Theorem \ref{thm:main}]
		If $\delta_\Omega\le 1$, then the estimate \eqref{in-eq:u-ustar} is a direct consequence of combining together Proposition \ref{est-pro:D-iso}, Theorem \ref{com-thm:Talenti} and Proposition \ref{pf-pro:u-ustar}. 
		
		Otherwise, when $\delta_\Omega> 1$, \eqref{in-eq:u-ustar} is clearly valid, as it is easy to show by arguing as in \eqref{pf-eq:2L1u}.
		
		The proof is complete.
	\end{proof}

	\medskip

	\subsection*{Acknowledgements}
	The authors have been supported by the Research Project of the Italian Ministry of University and Research (MUR) Prin 2022 ``Partial differential equations and related geometric-functional inequalities'', grant number 20229M52AS\_004. The first author has been also partially supported by the ``Gruppo Nazionale per l'Analisi Matematica, la Probabilit\`a e le loro Applicazioni'' (GNAMPA) of the ``Istituto Nazionale di Alta Matematica'' (INdAM, Italy).

\end{document}